\documentclass[12pt]{amsart}
\usepackage{amsfonts,amssymb,amscd,amsmath,enumerate,verbatim, fancyvrb}
\usepackage[latin1]{inputenc}
\usepackage{amscd}
\usepackage{latexsym}
\usepackage{mathptmx}
\usepackage{multicol}
\input xy
\xyoption{all}

\def\NZQ{\mathbb}               

\def\ZZ{{\NZQ Z}}

\def\frk{\mathfrak}               

\def\Phi{{\frk N}}
\def\opn#1#2{\def#1{\operatorname{#2}}} 
\opn\chara{char} 
\opn\length{\ell} 
\opn\pd{pd} 
\opn\rk{rk}
\opn\projdim{proj\,dim} 
\opn\injdim{inj\,dim} 
\opn\rank{rank}
\opn\depth{depth} 
\opn\grade{grade} 
\opn\height{height}
\opn\embdim{emb\,dim} 
\opn\codim{codim}

\opn\Tr{Tr} 
\opn\bigrank{big\,rank}
\opn\superheight{superheight}
\opn\lcm{lcm}
\opn\trdeg{tr\,deg}
\opn\reg{reg} 
\opn\lreg{lreg} 
\opn\ini{in} 
\opn\lpd{lpd}
\opn\size{size}
\opn\mult{mult}
\opn\dist{dist}
\opn\cone{cone}
\opn\lex{lex}
\opn\rev{rev}
\opn\div{div} \opn\Div{Div} \opn\cl{cl} \opn\Cl{Cl}
\opn\Spec{Spec} \opn\Supp{Supp} \opn\supp{supp} \opn\Sing{Sing}
\opn\Ass{Ass} \opn\Min{Min}
\opn\Ann{Ann} \opn\Rad{Rad} \opn\Soc{Soc}
\opn\Syz{Syz} \opn\Im{Im} \opn\Ker{Ker} \opn\Coker{Coker}
\opn\Am{Am} \opn\Hom{Hom} \opn\Tor{Tor} \opn\Ext{Ext}
\opn\End{End} \opn\Aut{Aut} \opn\id{id} \opn\ini{in}

\opn\nat{nat}
\opn\pff{pf}
\opn\Pf{Pf} \opn\GL{GL} \opn\SL{SL} \opn\mod{mod} \opn\ord{ord}
\opn\Gin{Gin}
\opn\Hilb{Hilb}\opn\adeg{adeg}\opn\std{std}\opn\ip{infpt}
\opn\Pol{Pol}
\opn\sat{sat}
\opn\Var{Var}
\opn\Gen{Gen}

\opn\aff{aff} \opn\con{conv} \opn\relint{relint} \opn\st{st}
\opn\lk{lk} \opn\cn{cn} \opn\core{core} \opn\vol{vol}
\opn\link{link} \opn\star{star}
\opn\gr{gr}

\def\pot#1#2{#1[\kern-0.28ex[#2]\kern-0.28ex]}

\opn\dirlim{\underrightarrow{\lim}}
\opn\inivlim{\underleftarrow{\lim}}
\def\Implies{\ifmmode\Longrightarrow \else
        \unskip${}\Longrightarrow{}$\ignorespaces\fi}
\def\implies{\ifmmode\Rightarrow \else
        \unskip${}\Rightarrow{}$\ignorespaces\fi}
\def\iff{\ifmmode\Longleftrightarrow \else
        \unskip${}\Longleftrightarrow{}$\ignorespaces\fi}

\let\:=\colon
\newtheorem{Theorem}{Theorem}[section]

\newtheorem{Remark}[Theorem]{Remark}

\newtheorem{Example}[Theorem]{Example}

\newtheorem{Question}[Theorem]{Question}

\let\epsilon\varepsilon
\let\phi=\varphi
\let\kappa=\varkappa
\def\qed{\ifhmode\textqed\fi
      \ifmmode\ifinner\quad\qedsymbol\else\dispqed\fi\fi}
\def\textqed{\unskip\nobreak\penalty50
       \hskip2em\hbox{}\nobreak\hfil\qedsymbol
       \parfillskip=0pt \finalhyphendemerits=0}
\def\dispqed{\rlap{\qquad\qedsymbol}}

\opn\dis{dis}
\opn\height{height}
\opn\dist{dist}
\def\pnt{{\raise0.5mm\hbox{\large\bf.}}}

\opn\Lex{Lex}

\begin{document}
\title{Lexsegment ideals and their h-polynomials}
\author{Takayuki Hibi and Kazunori Matsuda}
\address{Takayuki Hibi,
Department of Pure and Applied Mathematics,
Graduate School of Information Science and Technology,
Osaka University, Suita, Osaka 565-0871, Japan}
\email{hibi@math.sci.osaka-u.ac.jp}
\address{Kazunori Matsuda,
Kitami Institute of Technology, 
Kitami, Hokkaido 090-8507, Japan}
\email{kaz-matsuda@mail.kitami-it.ac.jp}
\subjclass[2010]{05E40, 13H10}
\keywords{Castelnuovo--Mumford regularity, lexsegment ideal, $h$-polynomial}
\begin{abstract}
Let $S = K[x_1, \ldots, x_n]$ denote the polynomial ring in $n$ variables over a field $K$ with each $\deg x_i = 1$ and $I \subset S$ a homogeneous ideal of $S$ with $\dim S/I = d$.  The Hilbert series of $S/I$ is of the form $h_{S/I}(\lambda)/(1 - \lambda)^d$, where $h_{S/I}(\lambda) = h_0 + h_1\lambda + h_2\lambda^2 + \cdots + h_s\lambda^s$ with $h_s \neq 0$ is the $h$-polynomial of $S/I$.  
Given arbitrary integers $r \geq 1$ and $s \geq 1$, a lexsegment ideal $I$ of $S = K[x_1, \ldots, x_n]$, where $n \leq \max\{r, s\} + 2$, satisfying $\reg(S/I) = r$ and $\deg h_{S/I}(\lambda) = s$ will be constructed.    
\end{abstract}
\maketitle
The study on the regularity and the degree of the $h$-polynomial of a monomial ideal done in \cite{HM} continues in the present paper and an affirmative answer to \cite[Conjecture 0.1]{HM} will be given.

Let $S = K[x_1, \ldots, x_n]$ denote the polynomial ring in $n$ variables over a field $K$ with each $\deg x_i = 1$ and $I \subset S$ a homogeneous ideal of $S$ with $\dim S/I = d$.  The Hilbert function of $S/I = \bigoplus_{n=0}^{\infty} (S/I)_n$ is the numerical function $H(S/I,n) = \dim_K (S/I)_n$ for $n \in \ZZ_{\geq 0}$.  The Hilbert series of $S/I$ is the formal power series $F(S/I, \lambda) = \sum_{n=0}^{\infty} H(S/I,n) \lambda^n \in \ZZ[[\lambda]]$ of $\{H(S/I,n)\}_{n=0}^{\infty}$.  It is known (\cite[Theorem 6.1.3]{hhGTM}) that $F(R/I, \lambda)$ is of the form 
\[
(h_0 + h_1\lambda + h_2\lambda^2 + \cdots + h_s\lambda^s)/(1 - \lambda)^d,
\] 
where each $h_i \in \ZZ$.  We say that 
\[
h_{S/I}(\lambda) = h_0 + h_1\lambda + h_2\lambda^2 + \cdots + h_s\lambda^s
\]
with $h_s \neq 0$ is the {\em $h$-polynomial} of $S/I$.  Let $\reg(S/I)$ denote the ({\em Castelnuovo--Mumford}\,) {\em regularity} \cite[p.~48]{hhGTM} of $S/I$.  In the previous paper \cite{HM}, given arbitrary integers $r$ and $s$ with $r \geq 1$ and $s \geq 1$, a monomial ideal $I$ of $S = K[x_1, \ldots, x_n]$ with $n \gg 0$ for which $\reg(S/I) = r$ and $\deg h_{S/I}(\lambda) = s$ was constructed and it is conjectured  that the desired monomial ideal can be chosen to be a strongly stable ideals (\cite[Conjecture 0.1]{HM}).  The purpose of the present paper is to claim that \cite[Conjecture 0.1]{HM} turns out to be true.  

\begin{Theorem}
\label{main}
Given arbitrary integers $r \geq 1$ and $s \geq 1$, there exists a lexsegment ideal $I$ of $S = K[x_1, \ldots, x_n]$ with $n \leq \max\{r, s\} + 2$ for which $\reg(S/I) = r$ and $\deg h_{S/I}(\lambda) = s$.  
\end{Theorem} 

\begin{proof} {\bf (First Step)}  Let $1 \leq r \leq s$.  Let $I$ be the lexsegment ideal \cite[p.~103]{hhGTM} generated by the monomials 
\[
x_1^{r+1}, \, \, x_1^{r}x_2, \, \, x_1^{r}x_3, \ldots, \, \, x_1^{r}x_{s-r+1}
\]
of the polynomial ring $S = K[x_1, \ldots, x_{s-r+1}]$.  Eliahou--Kervaire formula \cite[Corollary 7.2.3]{hhGTM} says that $I$ has a linear resolution and $\reg(S/I) = \reg(I) - 1 = r$.
Furthermore, since $I = x_1^{r}(x_1, x_2, \ldots, x_{s-r+1})$, one has
\begin{eqnarray*}
F(S/I, \lambda) & = & \frac{1}{(1 - \lambda)^{s-r+1}} - \left(\frac{\lambda^{r}}{(1 - \lambda)^{s-r+1}} - \lambda^r \right) \\
& = &
\frac{1 + \lambda + \cdots + \lambda^{r-1} + \lambda^r(1 - \lambda)^{s-r}}{(1 - \lambda)^{s-r}}.
\end{eqnarray*}
Hence $\deg h_{S/I}(\lambda) = s$, as desired.

\medskip

{\bf (Second Step)}  Let $1 \leq s < r$.  We introduce the sequence $\{a_n\}_{n=0}^{\infty}$ which is
\[
a_0 = 1, \, a_1 = \cdots = a_{s-1} = r + 2, \, a_s = \cdots = a_r = \cdots = r + 1. 
\]
Since the sequence $\{a_n\}_{n=0}^{\infty}$ satisfy the Macaulay condition \cite[Theorem 6.3.8]{hhGTM}, it follows that there exists a homogeneous ideal $I \subset S = K[x_1, \ldots, x_{r+2}]$ for which $H(S/I, n) = a_n$ for $n \in \ZZ_{\geq 0}$.  Let $I^{\rm lex} \subset S$ denote the lexsegment ideal for which the Hilbert function of $I^{\rm lex}$ coincides with that of $I$ (\cite[Theorem 6.3.1]{hhGTM}).  Now, if $j \geq r + 1$, then $j - r + 1 \geq 1$ and the Macaulay expansion (\cite[Lemma 6.3.4]{hhGTM}) of $h_j$ is
\[
h_j = r + 1 = {j \choose j} + {j-1 \choose j-1} + \cdots + {j - r + 1 \choose j - r + 1}.
\] 
Hence 
\[
h_{j}^{\langle j \rangle} = {j + 1 \choose j + 1} + {j \choose j} + \cdots + {j - r + 2 \choose j - r + 2} = r + 1 = h_{j+1}.  
\]
Furthermore, the Macaulay expansion of $h_r$ is
\[
h_r = r + 1 = { r + 1 \choose r}.
\] 
Hence
\[
h_{r}^{\langle r \rangle} = { r + 2 \choose r + 1} = r + 2 > h_{r+1}.
\]
It then follows from the proof of \cite[Theorem 6.3.1]{hhGTM} that the maximal degree of the monomials belonging to the minimal system of generators of $I^{\rm lex}$ is $r + 1$.  Again, Eliahou--Kervaire formula \cite[Corollary 7.2.3]{hhGTM} says that $\reg(R/I) = \reg(I) - 1 = (r + 1) - 1 = r$.  On the other hand,

\begin{eqnarray*}
F(S/I, \lambda) & = & 1 + (r + 2)(\lambda + \lambda^2 + \cdots + \lambda^{s-1}) + \frac{(r + 1) \lambda^{s}}{1 - \lambda} \\
& = & 1 + \frac{(r + 2)\lambda(1 - \lambda^{s-1})}{1 - \lambda} + \frac{(r + 1) \lambda^{s}}{1 - \lambda} \\
& = & \frac{1 + (r + 1)\lambda - \lambda^s}{1 - \lambda}
\end{eqnarray*}
Hence $\deg h_{S/I}(\lambda) = s$, as required.
\end{proof}

\begin{Example}
\label{EX}
{\em  Let $s = 2$ and $r = 4$.  Then the lexsegment ideal $I^{\rm lex} \subset S = K[x_1, \ldots, x_6]$ with the Hilbert function $1, 6, 5, 5, 5, \ldots$ is the monomial ideal generated by  
\begin{center}
$x_{1}^{2}, x_{1}x_{2}, x_{1}x_{3}, x_{1}x_{4}, x_{1}x_{5}, x_{1}x_{6}, x_{2}^{2}, x_{2}x_{3}, x_{2}x_{4}, x_{2}x_{5}, x_{2}x_{6}, $
$x_{3}^{2}, x_{3}x_{4}, x_{3}x_{5}, x_{3}x_{6},  x_{4}^{2}, x_{4}x_{5}^{2}, x_{4}x_{5}x_{6}, x_{4}x_{6}^{3}, x_{5}^{5}. $
\end{center}
One has $\dim S/I^{\rm lex} = 1$ and $\depth S/I^{\rm lex} = 0$.  Its Betti table is

\ 

\centering
\begin{BVerbatim}
1  .  .  .  .  . . 
. 16 47 63 46 18 3 
.  2  9 16 14  6 1 
.  1  5 10 10  5 1 
.  1  4  6  4  1 . 
\end{BVerbatim}

}
\end{Example}

\begin{Question}
\label{Q}
Find all possible sequences $(d,e,r,s) \in \ZZ_{\geq 0}$ with $d \geq e \geq 0, \, r \geq 1, \, s \geq 1$ and $s - r \le d - e$ for which
there exists a homogeneous ideal $I \subset S = K[x_1, \ldots, x_n]$ with $n \gg 0$ satisfying 
\[
\dim S/I = d, \, \, \, \depth S/I = e, \, \, \, \reg(S/I) = r, \, \, \, \deg h_{S/I}(\lambda) = s.
\]
\end{Question}
On the other hand, in general, it is known (\cite[Corollary B.4.1]{V}) that one has 
\[
\deg h_{S/I}(\lambda) - \reg(S/I) \le \dim S/I - \depth S/I
\]
for each homogeneous ideal $I$ of $S = K[x_1, \ldots, x_n]$.

\begin{Remark}
{\em
A few remarks are collected.
\begin{enumerate}
	\item Let $1 \leq s < r$. 
	Let $I^{\rm lex}$ be the lexsegment ideal of $S = K[x_{1}, \ldots, x_{r + 2}]$ 
	which appears in Second Step of the proof of Theorem \ref{main}. 
	Then $\dim S/I^{\rm lex} = 1$, $\depth S/I^{\rm lex} = 0$, 
	$\reg \left(S/I^{\rm lex}\right) = r$ and $\deg h_{S/I^{\rm lex}}(\lambda) = s$. 
	As a result, Question \ref{Q} can be solved, when $d = 1, e= 0$ and $1 \leq s < r$. 
	\item Let $1 \leq s < r$. 
	Let $I = I_{r - s} + (u_{1}u_{2} \cdots u_{s})$ be the monomial ideal of 
	$S$ which appears in the proof of \cite[Theorem 1.2]{HM}, 
	where $$S = K[x,y_{1}, \ldots, y_{r - s}, z_{1}, \ldots, z_{r - s + 1}, u_{1}, u_{2}, \ldots, u_{s}].$$
	Then $\dim S/I = r$, $\depth S/I= r - 1$, 
	$\reg (S/I) = r$ and $\deg h_{S/I}(\lambda) = s$. 
	As a result, Question \ref{Q} can be also solved, when
	$d = r, e = r - 1$ and $1 \leq s < r$. 
	\item Let $I$ be the monomial ideals of $S = K[x_{1}, x_{2}, x_{3}, x_{4}, x_{5}]$ generated by
\begin{eqnarray*}	
	x_{1}^{2}, \, x_{1}x_{2}, \, x_{1}x_{3}, \, x_{1}x_{4}, \, x_{1}x_{5}, \, x_{2}^2, \, x_{2}x_{3}^{2}, \, x_{2}x_{3}x_{4}, \, 
	x_{2}x_{3}x_{5},\\	
	x_{2}x_{4}^{3}, \, x_{2}x_{4}^{2}x_{5}, \, x_{2}x_{4}x_{5}^3, \,
	x_{2}x_{5}^{4}, \, 
	x_{3}^{6}, \, x_{3}^{5}x_{4}, \, x_{3}^{5}x_{5}, \, x_{3}^{4}x_{4}^{3}. \, \, \, \, \, \, \, 
\end{eqnarray*} 
	Then $I$ is a lexsegment ideal and $\dim S/I = 2$, $\depth S/I = 0$, $\reg (S/I) = 6$ and 
	$\deg h_{S/I}(\lambda) = 1$. 
	It would be of interest to find a lexsegment ideal $I$ of $S = K[x_{1}, \ldots, x_{n}]$ 
	with $n \gg 0$ satisfying $\dim S/I - \depth S/I = c$ and $\deg h_{S/I}(\lambda) < \reg (S/I)$ 
	for all $c \geq 0$. 
	\item Let $I$ be a monomial ideal of $S = K[x, y]$. 
	By virtue of \cite[Theorem 2.4]{CH}, one can easily see that 
	\begin{itemize}
		\item $S/I$ is Cohen--Macaulay if and only if $\deg h_{S/I}(\lambda) = \reg (S/I)$;
		\item $S/I$ is not Cohen-Macaulay if and only if $\deg h_{S/I}(\lambda) = \reg (S/I) + 1$.  
	\end{itemize}
	Thus in particular one has $\deg h_{S/I}(\lambda) \geq \reg (S/I)$ for all monomial ideals $I$ of $S = K[x, y]$. 
	\item It might be a reasonable question to find a natural class of monomial ideals $I$ of $S = K[x_1, \ldots, x_n]$ for which (i) $S/I$ is not Cohen--Macaulay, (ii) $I$ does not have a pure resolution and (iii) $\deg h_{S/I}(\lambda) - \reg(S/I) = \dim S/I - \depth S/I$. 	 
\end{enumerate}
}
\end{Remark}

\noindent
{\bf Acknowledgements} 
\, During the participation of the first author in the workshop {\em New Trends in Syzygies}, organized by Jason McCullough (Iowa State University) and Giulio Caviglia (Purdue University), Banff International Research Station for Mathematical Innovation and Discovery, Banff, Canada, June 24 -- 29, 2018, a motive for writing the present paper arose from an informal conversation with Marc Chardin.  Special thanks are due to the BIRS for providing the participants with a wonderful atmosphere for mathematics.   
The first author is partially supported by JSPS KAKENHI 26220701.  The second author is partially supported by JSPS KAKENHI 17K14165.

\end{document}